\documentclass[12pt,a4paper]{amsart}
\usepackage{amssymb,amscd,xypic}

\theoremstyle{plain}
\newtheorem{theorem}{Theorem}[section]

\newtheorem{proposition}[theorem]{Proposition}
\newtheorem{corollary}[theorem]{Corollary}

\theoremstyle{definition}
\newtheorem{definition}[theorem]{Definition}
\newtheorem{remark}[theorem]{Remark}

\newtheorem{example}[theorem]{Example}

\numberwithin{equation}{section}

\newcommand\cA{{\mathcal A}}

\newcommand\cO{{\mathcal O}}

\newcommand\aff{\operatorname{aff}}
\newcommand\ant{\operatorname{ant}}

\newcommand\Ker{\operatorname{Ker}}

\newcommand\Spec{\operatorname{Spec}}

\title{The ring of regular functions of an algebraic monoid}

\author{Lex Renner and Alvaro Rittatore}
\thanks{The first named author was partially supported by a grant from NSERC.
The second named author was partially supported by grants from IMU/CDE,
NSERC and PDT/54-02 research project}

\begin{document}

\begin{abstract}
Let $M$ be an irreducible normal algebraic monoid with unit group $G$.
It is known that $G$ admits a Rosenlicht decomposition,
$G=G_{\ant}G_{\aff}\cong (G_{\ant}\times G_{\aff})/G_{\aff}\cap
  G_{\ant}$, where  $G_{\ant}$ is the maximal anti-affine subgroup of
$G$, and $G_{\aff}$ the maximal normal connected affine subgroup of
$G$. In this paper we show that this decomposition extends to a
decomposition
$M=G_{\ant}M_{\aff}\cong G_{\ant}*_{G_{\aff}\cap
  G_{\ant}}M_{\aff}$, where $M_{\aff}$ is the affine submonoid
$M_{\aff}=\overline{G_{\aff}}$. We then use this decomposition  to calculate
$\mathcal{O}(M)$ in terms of $\mathcal{O}(M_{\aff})$ and
$G_{\aff}, G_{\ant}\subset G$. In particular, we determine when $M$ is
an anti-affine monoid, that is $\mathcal{O}(M)=\Bbbk$.
\end{abstract}

\maketitle

\section{Introduction}

The theory of {\em affine} algebraic monoids has been
investigated extensively over the last thirty years.
See \cite{Pu88, Re05, So93} for different accounts of these
developments. More recently there has been some important progress on
the structure of non-affine algebraic monoids. By generalizing a
classical theorem of Chevalley, the authors of \cite{ritbr} prove that
any normal algebraic monoid is an extension of an affine algebraic
monoid by an abelian variety. This allows one to analyze the structure
of such monoids in terms of more basic objects: affine monoids, abelian
varieties and anti-affine algebraic groups.

To state our results we first introduce some notation. Let
$\Bbbk $ be an algebraically closed field. We work
with algebraic varieties $X$ over $\Bbbk $, that is, integral,
separated schemes over $\Bbbk $. An
algebraic group is assumed to be a smooth group scheme of finite type
over $\Bbbk $. If $X$ is an algebraic variety we denote by
$\mathcal{O}(X)$ the ring of regular functions on $X$.
If $X$ is an affine variety and
$I\subset \cO(X)$ is an ideal, we denote by $\mathcal
V(I)=\bigl\{x\in X\mathrel{:} f(x)=0 \ \forall \, f\in I\bigr\}$;  if
$Y\subset X$ is a subset, we denote by $\mathcal I(Z)=\bigr\{ f\in
\cO(X)\mathrel{:} f(y)=0\ \forall\, y\in Y\bigr\}$.
If $X$ is
irreducible we denote
by $\Bbbk (X)$ the field of rational functions on $X$. If $A$ is
any integral domain we denote by $[A]$ its quotient field. Hence, if
$X$ is an irreducible affine variety then $\Bbbk(X)=\bigl[\cO(X)\bigr]$.

Let $M$ be a connected, normal, \emph{algebraic monoid} with unit group
$G$ (see Definition \ref{def:algmon} below). The original motivation for this paper
was to investigate the following basic question, first posed by M. Brion.
\[
\text{``How does one describe $\cO(M)$, and when is it finitely generated?"}
\]
Although we do not answer this question completely, we obtain many
remarkable results about $\cO(M)$.

Let $M$ and $G$ be as above.
By the results of \cite{ritbr}, if $\alpha_G : G \to A$ is the
unique Albanese morphism of $G$ such that $\alpha_G(1_G)=0_A$ (note the
additive notation for $A$), then there exists a unique morphism
$\alpha_M : M \to A$
such that $\alpha_M|_{_G}=\alpha_G$. Furthermore, $\alpha$ is an
affine morphism,
and the scheme-theoretic fibers
of $\alpha$ are normal varieties. The fiber at $1\in A$ is $M_{\aff}$,
the unique irreducible, affine submonoid of $M$ with unit group
$G_{\aff}$, the kernel of $\alpha$. See Theorem \ref{thm:strucmonoid} below.

The purpose of this paper is three-fold. First we identify
$\mathcal{O}(M)$ in terms of the structure of $M$ and
$M_{\aff}$, see Theorem \ref{thm:crucial}. We then identify
conditions under which
$\bigl[\mathcal{O}(M)\bigr]=\bigl[\mathcal{O}(G)\bigr]$, see Theorem
\ref{thm:stablefracfield}. Finally, we determine the
conditions
under which $\mathcal{O}(M)=\Bbbk $ (that is when $M$ is an
\emph{anti-affine} algebraic monoid, Theorem
\ref{thm:TheTheorem}). In order to establish our results we define the
notion of a \emph{stable} algebraic monoid (see Definition
\ref{def:stable}).

To obtain our main results we make use of the generalized Chevalley
decomposition presented in
\cite{ritbr}, which states that, if $M$ is an
irreducible monoid then
\[
M\cong G*_{G_{\aff}}M_{\aff},
\]
where $G_{\aff}$ is the smallest affine algebraic group such that
$G/G_{\aff}$ is an abelian variety (see Theorem \ref{thm:strucmonoid} below).
This structural result allows us to present a
Rosenlicht decomposition $M=G_{\ant}*_{G_{\aff}\cap
G_{\ant}}M_{\aff}$ that generalizes the corresponding decomposition
$G=(G_{\ant}\times G_{\aff})/(G_{\aff}\cap G_{\ant})$ of $G$,
where $G_{\ant}$ is the largest \emph{anti-affine}
subgroup of $G$. See \cite{kn:brionantiaff} and Proposition
\ref{prop:resultsM} below.
We then use this decomposition (of $M$)
in Theorem \ref{thm:crucial} to calculate
$\mathcal{O}(M)$ in terms of $\mathcal O(M_{\aff})$
and $H=G_{\aff}\cap G_{\ant}$.

Next we identify
a set of central idempotents $e$ of $M$ such that
$\mathcal{O}(M)=\mathcal{O}(eM)$. But such an idempotent can be chosen
so that $eM$ is a {\em stable} monoid. Consequently, we reduce ourselves
to the study of stable monoids, thereby obtaining a characterization of
the algebraic monoids $M$ such that $\mathcal O(M)=\Bbbk$, the
\emph{anti-affine} algebraic monoids. See Theorem
\ref{thm:TheTheorem}.

Let $M$ be an anti-affine algebraic monoid and let $e\in E(M)$ be the minimum
idempotent of $M$. In Theorem \ref{thm:genalb} we show that the retraction
$\ell_e : M\to eM$, $\ell_e(m)=em$, is Serre's universal morphism from $M$ to a
commutative algebraic group (see \cite[Thm.~8]{Se58}).

We conclude the paper with Theorem \ref{thm:rosenmon}.
Here we show that there is an analogue of the Rosenlicht decomposition
for a large class of normal, algebraic monoids. In particular the fibre $\varphi^{-1}(1)$,
of the canonical map $\varphi : M\to Spec(\mathcal O(M))$, is an anti-affine monoid
which we identify explicitly in terms of the internal structure of $M$.

\medskip

{\sc Acknowledgements: }
This paper was written during a stay of the second author at the
University of Western Ontario. He would like to thank them for the kind
hospitality he received during his stay.


\section{Preliminaries}

In this section we assemble some of what is known about algebraic monoids
with nonlinear unit groups. These results are due to M. Brion and the second
named author \cite{Br06,ritbr,kn:brionlocal,Ri06}.

\begin{definition}
\label{def:algmon}
An {\em algebraic monoid}\/ is an algebraic variety $M$
together with a morphism  $m:M\times M\to M$ such that $m$ is an
associative product and there exists a neutral element $1\in M$.
The {\em unit group}\/ of $M$ is the group of invertible elements
\[
G(M)=\left\{ g\in M\mathrel{:} \exists\,
g^{-1}\,,\, gg^{-1}=g^{-1}g=1\right\}.
\]
We denote the \emph{set of idempotent elements} by $E(M)=\{e\in
M\mathrel{:} e^2=e\}$.
\end{definition}

It has been  proved that $G(M)$ is an algebraic group, open in $M$ (see
for example \cite{Ri98}). The  structure of  $M$ is significantly
influenced by the  structure of $G(M)$. For
instance, the $G(M)\times G(M)$ action by left and right
multiplication,  $(a,b)\cdot m=amb^{-1}$, has $G(M)$ as open orbit
and, if $G(M)$ is dense in $M$, an unique closed orbit, which is the
\emph{Kernel} of $M$,
i.e.~the minimum, closed, nonempty subset $I$ such that $MIM\subset I$.

Recall that  if $G$ is  algebraic group, then the Albanese
morphism  $p:G\to A(G)$ fits into an exact sequence
\begin{center}
\mbox{
\xymatrix{
1\ar@{->}[r]&G_{{\aff}}\ar@{->}[r]&
G\ar@{->}[r]^-p &A(G) \ar@{->}[r] & 0
}
}
\end{center}
where $G_{\aff}$ is a normal connected affine algebraic group (since the
group $\cA(G)$ is commutative, its law will be denoted
additively). Moreover, $G_{\aff}$ is the smallest affine algebraic
subgroup such that $G/G_{\aff}$ is an abelian variety. This structure
theorem is originally due to Chevalley, but now there is a modern proof
in \cite{Co02}. Recently it has been generalized from groups to monoids.
The following theorem is a summary of this development.

\begin{theorem}[Brion, Rittatore {\cite{Br06,ritbr,kn:brionlocal,Ri06}}]
\label{thm:strucmonoid}
Let $M$ be a normal irreducible algebraic monoid with unit group
$G$. Then $M$ admits a Chevalley decomposition:

\begin{center}
\mbox{
\xymatrix{
1\ar@{->}[r]&M_{{\aff}}=\overline{G_{{\aff}}}\ar@{->}[r]&
M=\overline{G}\ar@{->}[r]^-p
  &A(G) \ar@{->}[r] &0\\
1\ar@{->}[r]&G_{{\aff}}\ar@{->}[r]\ar@{^(->}[u]&
G\ar@{->}[r]_-p\ar@{^(->}[u] &A(G) \ar@{->}[r]\ar@{=}[u] & 0
}
}
\end{center}
where $p:M\to A(G)=G/G_{{\aff}}$ and $p|_{_G}:G\to A(G)$ are the Albanese
morphisms  of $M$ and $G$ respectively, and $M_{\aff}$ is an
affine algebraic monoid.
\qed
\end{theorem}

\begin{definition}
Let $ G$ be an algebraic group, and let a closed subgroup $H\subset G$
act on
an algebraic variety $X$. The {\em induced space}\/ $G*_HX$
is defined as the geometric quotient
of $G\times X$ under the $H$-action $h\cdot(g,x)=(gh^{-1},h\cdot
x)$.

Under mild conditions on $X$ (e.g.~$X$ normal and covered by
quasi-projective $H$-stable open subsets), this quotient exists. Clearly,
$G*_HX$ is a $G$-variety, for the action induced by $a\cdot
(g,x)=(ag,x)$. We will denote the class of $(g,x)$ in $G\times X$ by
$[g,x]\in G*_HX$. The fundamental properties of
$G*_HX$ were established by Bialynicki-Birula in \cite{kn:bb-induced}.
See also \cite{Ti06}.
\end{definition}

\begin{remark}
Let $G$ be an algebraic group, and $H\subset G$ a closed subgroup acting over an
algebraic variety $X$. Then $\pi:G*_HX\to G/H$, which is induced by
$(g,x)\mapsto gH$, is a fiber bundle over $G/H$, with fiber
isomorphic to $X$.
\end{remark}

\begin{theorem}[{Brion, Rittatore, \cite{Br06,ritbr}}]
\label{thm:brionrittabel}
Let $M$ be a normal algebraic monoid and let $Z^0$ be the connected
center of $G$. Then $A(G)\cong
Z^0/(Z_0\cap
G_{{\aff}})$ and
\[
M=GM_{{\aff}}=Z^0M_{{\aff}}\cong G*_{G_{{\aff}}}M_{{\aff}}\cong  Z^0*_{Z_0\cap
  G_{{\aff}}}M_{{\aff}}.
\]
\end{theorem}

\begin{definition}
If $M$ is an algebraic monoid with unit group, we define the \emph{center
of $M$} be
\[
\mathcal Z(M)=\bigl\{ z\in M\mathrel{:} zm=mz\ \forall\, m\in M\bigr\},
\]
the set of  central elements.
\end{definition}

It is clear that  $\mathcal Z(M)$ is a closed
submonoid of $M$, with unit group $G\bigl(\mathcal Z(M)\bigr)= \mathcal
Z(G)$, the center of $G$.  However, one should be aware that this monoid
is not necessarily connected.   Moreover, the following example shows
that the $\mathcal Z(G)$ is not necessarily dense in $\mathcal Z(M)$.

\begin{example}
\label{exam:center}
 Let $r,s\in \mathbb N$, $r\neq s$, and consider the affine
  algebraic monoid
\[
 N=\bigl\{\left( \begin{smallmatrix}
t^r & a\\
0& t^s
\end{smallmatrix}
\right)
\mathrel{:}
t,a\in \Bbbk
\bigr\}.
\]
Then $G(N)= \bigl\{\left( \begin{smallmatrix}
t^r & a\\
0& t^s
\end{smallmatrix}
\right)
\mathrel{:}
t\in \Bbbk^* \,,\ a\in \Bbbk
\bigr\}$. The center of $G(N)$ is the finite subgroup
$\mathcal Z(G)=\bigl\{\left( \begin{smallmatrix}
t^r & 0\\
0& t^s
\end{smallmatrix}
\right)
\mathrel{:}
t^r=t^s
\bigr\}$. The center of $N$ is $\mathcal Z(N)= \mathcal Z(G)\cup
\{0\}$. However $\overline{\mathcal Z(G)}\neq \mathcal Z(N)$.

In particular the zero matrix is a central idempotent which does not
belong to $\overline{\mathcal Z(G)}$.
\end{example}

In what follows  we collect some results about the \emph{Rosenlicht
decomposition} of an algebraic group $G$. This decomposition depicts $G$
as the product of two subgroups, one affine
and the other anti-affine. We refer the reader to \cite{Ro56},
\cite[Sec.~III.3.8]{DG70} and \cite{kn:brionantiaff} for proofs and
further results about this decomposition.

\begin{definition}
A connected  algebraic group $G$ is \emph{anti-affine} if
$\cO(G)=\Bbbk$.
\end{definition}

\begin{remark}
It is easy to see that any anti-affine group is commutative.
See for  example \cite[Lem.~1.1]{kn:brionantiaff}.
\end{remark}

\begin{theorem}[Rosenlicht decomposition]
Let $G$ be a connected algebraic group. Then $\cO(G)$ is a finitely
generated algebra, in such a way that $\Spec\bigl(\cO(G)\bigr)$ is an
affine algebraic group.

Let $\varphi_G:G\to    \Spec\bigl(\cO(G)\bigr)$ be the canonical
morphism (the \emph{affinization}). Then $G_{\ant}=\Ker(\varphi_G)$ is
a connected subgroup, contained in the center of $G$.
The subgroup $G_{\ant}$ is the largest anti-affine subgroup scheme of
$G$. Equivalently, $G_{\ant}$ is the smallest normal subgroup scheme of
$G$ such that $G/G_{\ant}$ is affine.

 Moreover,
 $G=G_{\aff}G_{\ant}\cong (G_{\aff}\times G_{\ant})/(G_{\aff}\cap
 G_{\ant})$, and $G_{\aff}\cap G_{\ant}$ contains $(G_{\ant})_{\aff}$
 as an algebraic group of finite index.
\qed
\end{theorem}

In loose terms, an anti-affine algebraic group $G$ is a non-split
extension of an
abelian variety by an affine, commutative algebraic group. See
\cite[Thm.~2.7]{kn:brionantiaff}. But in the case
$\operatorname{char}(\Bbbk)=p>0$ the situation
is considerably less complicated. Indeed we have the following
simplifying result.
See  \cite[Prop.~2.2]{kn:brionantiaff}.

\begin{proposition}[Brion \cite{kn:brionantiaff}]
\label{prop:antiaffchar}
Let $G$ be an anti-affine, connected algebraic group over an
algebraically closed field $\Bbbk$. If $\operatorname{char}(\Bbbk)=p>0$
then $G$ is a semi-abelian variety.
i.e.~$G$ is the extension of an abelian variety by an affine torus group.
\begin{center}
\mbox{
\xymatrix{
1\ar@{->}[r]&T\ar@{->}[r]&
G\ar@{->}[r]^-p &A(G) \ar@{->}[r] & 0.
}
}
\end{center}
\end{proposition}


The following result generalizes Rosenlicht's decomposition to the
case of algebraic monoids. It is essential for determining the ring of
regular functions on an algebraic monoid.

\begin{proposition}[Rosenlicht decomposition for $M$]
\label{prop:resultsM}
Let $M$ be a normal irreducible algebraic monoid with unit group $G$. Then
$M \cong G_{\ant}*_{G_{\aff}\cap G_{\ant}}M_{\aff}$ so that
\[
M = G_{\ant}M_{\aff} \cong G_{\ant}*_{G_{\aff}\cap G_{\ant}}M_{\aff}.
\]
\end{proposition}

\begin{proof}
Since $G=G_{\ant}G_{\aff}$, it follows that $M=GM_{\aff}=G_{\ant}M_{\aff}$.
The morphism $G_{\ant}\times M_{\aff}\to M$ $(g,m)\mapsto gm$ induces a
surjective morphism of algebraic monoids
  $\varphi: G_{\ant}*_{G_{\aff}\cap G_{\ant}} M_{\aff}\to M$,
$\varphi\bigl([g,m]\bigr)=gm$.  If $a,b\in G_{\ant} $ and $m,n\in M_{\aff}$
are such that $am=bn$, then $b^{-1}am=n$. It follows from Theorem
\ref{thm:brionrittabel} that  $b^{-1}a\in
G_{\aff}\cap G_{\ant}$, and thus $[a,m]=[b,n]$;
that is, $\varphi$ is injective. Since
$\varphi|_{_G}:G=G*_{G_{\aff}}G_{\aff}\to G$ is an isomorphism and that
$M$ is normal, it follows from Zariski's main Theorem that
$\varphi $ is an isomorphism.
\end{proof}

\begin{corollary}
Let $M$ be an irreducible, normal, algebraic monoid such that
$M_{\aff}$ is a monoid with  zero
$0$. Then $0M$ is the Kernel of $M$. Moreover, $0
M$ is an algebraic group, and
\[
0M=0G=0G_{\ant}\cong
G_{\ant}/(G_{\aff}\cap G_{\ant}).
\]
\end{corollary}

\begin{proof}
By Proposition \ref{prop:resultsM}, we have that
$0M=0G_{\ant}M_{\aff}=0G_{\ant}$. In particular,
$0M$ is the image under the geometric quotient $G_{\ant}\times
M_{\aff}\to M=G_{\ant}*_{G_{\aff}\cap G_{\ant}}M_{\aff}$ of the closed
$G_{\ant}$-orbit $G_{\ant}\times \{0\}$, and hence is closed in
$M$. Since $0$ is a central idempotent of $M$,  it follows that
$0M$ is the unique $(G\times G)$-closed orbit of $M$. Thus, by
\cite[Theorem 1]{Ri98}, $0M$ is the Kernel of $M$.

Since $0$ is central, it follows that
$0M=0G_{\ant}$ is an algebraic group.
Consider the
multiplication morphism
$\ell:G_{\ant}\to 0G_{\ant}$, $\ell(g)= 0g$. Then
\[
\ell^{-1}(0)=
\{g\in G_{\ant}\mathrel{:} 0g=0\}= (G_{\ant})_{0}.
\]

It follows that $g\in \ell^{-1}(0)$ if and only if
$[g,0]=[1,0]\in  G_{\ant}*_{G_{\aff}\cap
  G_{\ant}}M_{\aff}$; that is, if and only if $g\in G_{\aff}\cap
  G_{\ant}$. Since $\ell$ is a separable morphism, it follows that
  $0G_{\ant}\cong  (G_{\aff}\cap G_{\ant})$.
\end{proof}

\section{The algebra of regular functions of $M$}

The following Theorem is the key to understanding the ring of
regular functions on an algebraic monoid.

\begin{theorem}
\label{thm:crucial}
Let $M$ be a normal algebraic monoid. Then $\cO(M)\cong
\cO(M_{\aff})^{G_{\aff}\cap G_{\ant}}$.
\end{theorem}

\begin{proof}
By Proposition \ref{prop:resultsM}, it follows that $M\cong
G_{\ant}*_{G_{\aff}\cap G_{\ant}} M_{\aff}$, and hence $G_{\ant}\times
M_{\aff} \to M$ is a geometric
quotient. Thus
\[
\begin{split}
\cO(M)= & \ \cO\bigl(G_{\ant}\times M_{\aff} \bigr)^{G_{\ant}\cap
  G_{\aff}}= \\
& \ \bigl( \cO(G_{\ant})\otimes \cO(M_{\aff})\bigr)^{G_{\ant}\cap
  G_{\aff}}=\\
& \ \cO(M_{\aff})^{G_{\ant}\cap
  G_{\aff}} ,
\end{split}
\]
where for the last equality we used that $\cO(G_{\ant})=\Bbbk$.
\end{proof}

\begin{corollary}
Assume that $\operatorname{char}(\Bbbk)=p>0$. Then $\cO(M)$ is a
finitely generated algebra.
\end{corollary}
\begin{proof}
Indeed, by Proposition \ref{prop:antiaffchar}, it follows that
$(G_{\aff}\cap G_{\ant})^0$, the connected component of the identity,
is a torus, and hence
$\cO(M)=\cO(M_{\aff})^{G_{\aff}\cap G_{\ant}}$ is a finitely generated algebra.
\end{proof}

\begin{theorem}
\label{thm:eM}
Let $M$ be a normal algebraic monoid and let $e\in E(M)$ be a central
idempotent.
Let $M_e=\{m\in M\mathrel{:} me=e\}$ and let $G_e=\{g\in G(M)\mathrel{:}
ge=e\}$. Then
\noindent (1) $M_e=\overline{G_e}$, and $M_e$ is an irreducible
algebraic monoid, with
unit group $G_e\subset G_{\aff}$. In particular, $G_e$ and $M_e$ are affine.

\medskip

\noindent (2) The subset $eM\subset M$ is an algebraic monoid, closed in
  $M$. The morphism $\ell_e :M\to eM$, $m\mapsto em$,  is a morphism
  of algebraic
  monoids, with  $M_e=\ell_e^{-1}(1)$.

\medskip

\noindent (3) The unit group of $eM$ equals $G(eM)=eG$. Moreover,
  $(eG)_{\aff}=e(G_{\aff})$  and $(eG)_{\ant}=(eG_{\ant})$.
In particular, the Chevalley decompositions of $eM$ and $eG=G(eM)$
fit into the following commutative diagram of exact sequences
\begin{center}
\mbox{
\xymatrix{
1\ar@{->}[r]&
eM_{{\aff}}\ar@{->}[r]&
eM\ar@{->}[r]^-p
  &A(G) \ar@{->}[r] &0\\
1\ar@{->}[r]&eG_{{\aff}}\ar@{->}[r]\ar@{^(->}[u]&
eG\ar@{->}[r]_-p\ar@{^(->}[u] &A(G) \ar@{->}[r]\ar@{=}[u] & 0
}
}
\end{center}
where $eM_{{\aff}}=\overline{eG_{{\aff}}}=e\overline{G_{{\aff}}}$, and
$eM=\overline{eG}$.

\medskip

\noindent (4) $eM= G_{\ant}eM_{\aff}\cong G_{\ant}*_{G_{\aff}\cap
    G_{\ant}}eM_{\aff}\cong eG_{\ant}*_{eG_{\aff}\cap
    eG_{\ant}}eM_{\aff}$.
\end{theorem}

\begin{proof}
(1) Since $M$ is normal at $e$, it follows by
\cite[Corollary 2.2.5]{kn:brionlocal} that
$M_e$ is an irreducible algebraic monoid, with unit group $G_e$.

Since $M\cong G*_{G_{\aff}}M_{\aff}$, it follows that if $ ge=e$, then
$[g,e]= [1,e]$, and hence  $g\in
G_{\aff}$, i.e.\ $G_e\subset G_{\aff}$.

(2) Since $eM=\{x\in M
\mathrel{:} xe=e\}$, it is clear that  $eM$ is a closed subset. Hence,
$eM$ is an algebraic monoid and  $\ell_e: M\to eM$, $\ell_e(m)= em$, is
a morphism of
algebraic monoids.

(3)
Since $\ell_e: M\to eM$ is a surjective morphism of algebraic
monoids, it follows that $G(eM)=eG$. In particular, $G\to eG$ is a
surjective morphism of algebraic groups, and hence,  by \cite[Lemma
1.5]{kn:brionantiaff},   $eG_{\ant}\subset
(eG)_{\ant}$. It is clear that  $eG_{\aff}\subset (eG)_{\aff}$ and
$(eG_{\aff})(eG_{\ant})=eG$. Hence, $eG/eG_{\ant}\cong
eG_{\aff}/(eG_{\aff}\cap eG_{\ant})$ is an affine algebraic
group, since $eG_{\aff}\cap eG_{\ant}$ is a central subgroup of
$eG_{\aff}$. It follows that $(eG)_{\ant}\subset eG_{\ant}$.
On the other hand, it is clear that $eG_{\aff}$ is a normal subgroup of
$eG$, and  the morphism $G\to eG/eG_{\aff}$, $g\mapsto
eg(eG_{\aff})$, induces a surjective morphism $G/G_{\aff}\to
eG/eG_{\aff}$. It follows that   $eG/eG_{\aff}$ is an abelian variety
and hence $eG_{\aff}=(eG)_{\aff}$.

Finally, since $eG\cong G/G_e$, and $G_e\subset G_{\aff}$, it
follows that $\mathcal A(eG)=eG/(eG_{\aff})\cong G/G_{\aff}=\mathcal
A(G)$, with the Albanese morphism fitting into the following
commutative diagram
\begin{center}
\mbox{
\xymatrix{
G\ar@{->}[r]^-{\ell_e}\ar@{->}[d]_-{\alpha_{eG}}&
eG\ar@{->}[d]^-{\alpha_{eG}}\\
\mathcal A(G)=G/G_{\aff}&\mathcal
A(eG)=eG/(eG)_{\aff}\ar@{->}[l]^-{\varphi}
}
}
\end{center}
where $\varphi: eG/(eG)_{\aff}\cong (G/G_e)/(G/G_e)_{\aff}\to
G/G_{\aff}$ is the
canonical isomorphism obtained by observing that $(G/G_e)_{\aff}=G_{\aff}/G_e$.

(4) follows from the description above and Proposition \ref{prop:resultsM}.
\end{proof}

\begin{remark}
Let $M$ be a normal  affine algebraic monoid and let $e\in E(M)$ be a central
idempotent. Then $\ell_e: M\to M$, $\ell_e(m)= em$, is such that
$\ell_e\circ \ell_e=\ell_e$. In other words, $\varphi^* : \cO(eM)\to
\cO(M)$ is a section  for the canonical surjection
$\cO(M)\to \cO(eM)$.
\end{remark}

\begin{corollary}
\label{coro:thmcrucial}
Let $M$ be a normal algebraic monoid and let $e\in E(M)$ be a central
idempotent. Then
$\cO(M)= \cO(eM)\oplus \mathcal I(eM_{\aff})^{G_{\ant}\cap G_{\aff}}$. In
particular, if $e\in E\bigl(\overline{G_{\ant}\cap G_{\aff}}\bigr)$,
then $\cO(M)=\cO(eM)$.
\end{corollary}

\begin{proof}
Indeed, by Theorem \ref{thm:crucial}, it follows that
\[
\begin{split}
\cO(M)=&\ \cO(M_{\aff})^{G_{\ant}\cap G_{\aff}}=\\
&\ \bigl(\cO(eM_{\aff})\oplus \mathcal
I(eM_{\aff})\bigr)^{G_{\ant}\cap G_{\aff}}= \\
&\ \cO(eM_{\aff})^{G_{\ant}\cap
  G_{\aff}}\oplus \mathcal
I(eM_{\aff})^{G_{\ant}\cap G_{\aff}}= \\
&\ \cO(eM)\oplus \mathcal
I(eM_{\aff})^{G_{\ant}\cap G_{\aff}}.
\end{split}
\]

Assume now that $e\in E\bigl(\overline{G_{\ant}\cap G_{\aff}}\bigr)$ and
let $f\in \cO(eM_{\aff})^{G_{\ant}\cap G_{\aff}}$. If $x\in M$, then
$f(x)=f(g\cdot x)$ for all $g\in  G_{\ant}\cap G_{\aff}$. It follows
that $f(x)=f(ex)$, since $e\in \overline{G_{\ant}\cap G_{\aff}}$.
In particular, if $f\in \mathcal I(eM_{\aff})^{G_{\ant}\cap G_{\aff}}$,
then $f(x)=f(ex)=0$ for all $x\in M$.
\end{proof}

\begin{remark}
\label{rem:idemGant}
The reader should notice that
\[
E\bigl(\overline{G_{\ant}\cap G_{\aff}}\bigr)=E\bigl(\overline{G_{\ant}} \bigr).
\]
Indeed, it follows from \cite[Prop.~3.1]{kn:brionantiaff} that
$(G_{\ant})_{\aff}\subset G_{\aff}\cap G_{\ant}$. Since $E(N)\subset
N_{\aff}$  for any algebraic monoid (\cite[Cor.~2.4]{ritbr}), it
follows that $E\bigl(\overline{G_{\ant}} \bigr)\subset
E\bigl(\overline{G_{\ant}\cap G_{\aff}}\bigr)$, the other inclusion
being obvious.
\end{remark}

\begin{definition}
An algebraic monoid $M$ is \emph{anti-affine} if $\cO(M)=\Bbbk$.
\end{definition}

Let $M$ be an algebraic monoid such that $G(M)$ is an
anti-affine algebraic group. Then $M$ is anti-affine. The
converse is not true, as the following example shows.

\begin{example}
Let $T=\Bbbk^*\times \Bbbk^*$ be an algebraic torus of dimension $2$,
and consider the affine toric variety $T\subset \mathbb A^2$. Then $\mathbb
A^2$ is an affine algebraic monoid with unit group $T$. Let $A$ be a
non-trivial connected abelian  variety and consider an extension
\begin{center}
\mbox{
\xymatrix{
0\ar@{->}[r] &\Bbbk^*=T_1  \ar@{->}[r] & H  \ar@{->}[r] &A \ar@{->}[r]&0
}
}
\end{center}
of algebraic groups such that $H$ is anti-affine.
Let $G=(T\times H)/T_1$, where $T_1\hookrightarrow T\times H$,
$t\mapsto (t,t)$. Then $G_{\aff}=T$, $G_{\ant}=H$, and $G_{\aff}\cap
G_{\ant}=T_1= \bigl\{  (t,t) \mathrel{:} t\in \Bbbk^*\bigr\}\subset T$.

The quotient $M=(\mathbb A^2\times H)/T_1$ is an
algebraic monoid with unit group $G$ and with $M_{\aff}\cong \mathbb A^2$. Thus
\[
\cO(M)=\cO(M_{\aff})^{G_{\aff}\cap
G_{\ant}}= \Bbbk[x,y]^{T_1}=\Bbbk.
\]
Hence, $M$ is an  anti-affine algebraic monoid while
$G(M)$ is not an anti-affine algebraic group.
\end{example}

\begin{definition}
\label{def:stable}
Let $G$ be an algebraic group and $X$ be a $G$-variety. We say that the
action is  \emph{generically stable} (equivalently that $X$ is a
\emph{generically stable
  $G$-variety}) if there exists an open subset consisting of  closed orbits.
We say that an algebraic monoid $M$ is \emph{stable} if it is
generically stable as a $(G_{\aff}\cap G_{\ant})$-variety.
\end{definition}

\begin{remark}
(1) Recall that any  regular action is such that there exists an open subset of
orbits of maximal dimension. Hence, an action $G\times X\to X$  is
generically stable if and only if the set of orbits of maximal dimension contains
an open subset consisting of closed orbits.

\noindent (2) If $M$ is an algebraic monoid, then $M$  is stable if and only if
$G_{\aff}\cap G_{\ant}$ is closed in $M$. Indeed, the coset
$g(G_{\aff}\cap G_{\ant})$, $g\in G$, is closed in $M$ if and only if
$G_{\aff}\cap G_{\ant}$ is closed in $M$.
\end{remark}

\begin{definition}
Let $G$ be an affine algebraic group acting on an affine variety
$X$. We say that the action is \emph{observable} if for every
non-zero $G$-stable ideal $I\subset \cO(X)$, $I^G\neq(0)$. Here we
consider the induced action $G\times \cO(X)\to\cO(X)$, $(a\cdot
f)(x)=f(a^{-1}x)$, for all $a\in G$, $x\in X$, $f\in \cO(X)$.
\end{definition}

The concept of observable action is a generalization of the notion of
observable subgroup. Observable subgroups were introduced by
Bialynicki-Birula, Hochschild and  Mostow   in \cite{kn:obsdef} and
have been researched extensively since then,
notably by F.~Grosshans (see \cite{kn:Gross-14} for a survey on
this topic). Given an affine algebraic group $G$, a closed subgroup
$H\subset G$ is said to be \emph{observable} if $G/H$ is a quasi-affine
algebraic variety. The equivalent definition of $H$ being observable
if every nonzero $H$-stable ideal $I\subset \cO(G)$ has the property
that $I^G\neq (0)$, was first recorded in \cite{fer-ritt}, and then
further generalized in \cite{kn:oaoag}. We present
here some of the basic results that we need in what follows. We
include some of the proofs here for convenience.

\begin{theorem}
\label{thm:obsercharac}
Let $G$ be an affine group acting on an affine variety $X$. Then the
action is observable if and only if
(1) $\bigl[\cO(X)\bigr]^G= \bigl[\cO(X)^G\bigr]$ and (2) $\Omega(X)$
contains a nonempty open subset.
\end{theorem}

\begin{proof}
We will only prove that if the action is observable then conditions (1)
and (2) hold. We refer the reader to \cite[Thm.~3.10]{kn:oaoag} for a
complete proof.

Clearly $\bigl[\cO(X)^G\bigr]\subset \bigl[\cO(X)\bigr]^G $. Let $g\in
\bigl[\cO(X)\bigr]^G$, and consider the ideal $I=\bigl\{ f\in
\cO(X)\mathrel{:} fg\in \cO(X)\bigr\}$. Clearly $I$ is $G$-invariant,
and hence there exists $ f\in \cO(X)^G$ such that $fg\in\cO(X)^G$.

Let now $X_{\max}$ be the (open) subset of orbits of maximal dimension and let
$Y=X\setminus X_{\max}$. Let $0\neq f\in \mathcal I (Y)^G$. Then the affine open
subset $X_f$ is is $G$-stable subset contained in $X_{\max}$. Let
$\cO\subset X_f$ be an orbit. Since every
orbit is open in its closure, it follows that if
$\overline{\cO}\neq \cO$ then $\overline{\cO}\cap Y\neq
\emptyset$. Since $f$ is constant on the orbits, it follows that
$f|_{_\cO}=0$ and hence $\cO\subset Y$, that is a contradiction.
\end{proof}

In \cite{kn:oaoag} the reader can find examples showing that both
conditions (1) and (2) of Theorem \ref{thm:obsercharac}   are
necessary. However, for some families of algebraic groups, generic
stability of the action implies observability. This is
the case when $G$ is a reductive group (see \cite[Thm.~4.7]{kn:oaoag}),
or when $G=U\times L$, where $L$ is reductive and $U$ is unipotent.

\begin{proposition}
\label{prop:partialconv}
Let $G$ be an affine algebraic group with Levi decomposition
$G=L\times U$, as above. Assume that $X$ is a generically stable
affine $G$-variety. Then the action is observable.
\end{proposition}

\begin{proof}
Let $I\subsetneq \cO(X)$ be a non-zero $G$-stable ideal. Then
$\mathcal V(I)\neq \emptyset $ is a proper closed subset. Since there
exists a open subset of  closed orbits, then there exists a closed orbit
$Z$ such that $Z\cap \mathcal V(I)=\emptyset$. Since $L$ is reductive
and that $Z$ is $L$-stable,
it follows that there exists $f\in \cO(X)^L$ such that $f\in I$,
$f|_{_{Z}}=1$. Hence, $I^L\neq (0)$. Since $U$ normalizes $L$,
it follows that $I\cap \cO(X)^L\neq \{0\}$ is an $U$-submodule, and hence
$I^G=(I^L)^U\neq (0)$.
\end{proof}

\begin{proposition}
\label{prop:antiaffstable}
Let $M$ be a  stable algebraic monoid, with $G(M)$ an anti-affine
algebraic group. Then $M=G(M)$.
\end{proposition}

\begin{proof}
If $G=G(M)$ is anti-affine, then $G_{\aff}= G_{\aff}\cap G_{\ant}$, and
since $M$ is stable, it follows that $M_{\aff}=\overline{G_{\aff}} =
\overline{ G_{\aff}\cap G_{\ant}}= G_{\aff}\cap G_{\ant}=
G_{\aff}$. Hence $M=GM_{\aff}=GG_{\aff}=G$.
\end{proof}

\begin{theorem}
\label{thm:stablefracfield}
Let $M$ be a stable normal algebraic monoid. Then
$\bigl[\cO(M)\bigr]=\bigl[\cO(G)\bigr]$.
\end{theorem}

\begin{proof}
Theorem \ref{thm:crucial} guarantees that   $\bigl[\cO(M)\bigr]=
\bigl[\cO(M_{\aff})^{G_{\aff}\cap G_{\ant}}\bigr]$. Since $M$
is stable then, by Proposition \ref{prop:partialconv} and Theorem
\ref{thm:obsercharac},
 it follows that
$\bigl[\cO(M)\bigr]= \bigl[\cO(M_{\aff})\bigr]^{G_{\aff}\cap
  G_{\ant}}$. But $M_{\aff}$ and $G_{\aff}$ being affine varieties, it
follows that
$\bigl[\cO(M)\bigr]=\bigl[\cO(G_{\aff})\bigr]^{G_{\aff}\cap
  G_{\ant}}$. Now,  applying the same reasoning to the algebraic
group $G$, we obtain that $\bigl[\cO(M)\bigr]=
\bigl[\cO(G_{\aff})^{G_{\aff}\cap G_{\ant}}\bigr]=\bigl[\cO(G)\bigr]$.
\end{proof}

\begin{corollary}
\label{coro:stableantiaff}
Let $M$ be an stable monoid. If  $M$ is
anti-affine, then $M$ is an (anti-affine) algebraic group.
\end{corollary}

\begin{proof}
By Theorem \ref{thm:stablefracfield}, it follows that $G=G(M)$ is an
anti-affine algebraic group.  Hence, by Proposition
\ref{prop:antiaffstable}, it follows that $M=G$.
\end{proof}

\begin{proposition}
\label{prop:existseMstab}
Let $M$ be a normal  algebraic monoid and let $e\in E(M)$ be a central
idempotent.  Then $eM$ is
stable if and only if $ef=e$ for all $f\in
E\bigl(\overline{G_{\ant}}\bigr)$. In particular, $eM$ is stable if
and only if
$ef_0=e$, where $f_0\in  E\bigr(\overline{G_{\ant}}\bigr)$ is the
minimum idempotent.
   \end{proposition}

\begin{proof}
We begin by recalling that
$E\bigr(\overline{G_{\ant}}\bigr)=E\bigr(\overline{G_{\aff}\cap
  G_{\ant}}\bigr)$ (see Remark \ref{rem:idemGant}).
Let $e\in E(M)$ be a central idempotent. It follows from  Theorem
\ref{thm:eM} that $G(eM)=eG$, and that $(eG)_{\aff}\cap
(eG)_{\ant}=e(G_{\aff}\cap G_{\ant})$. Since  $\ell_e|_{\overline{G_{\aff}\cap
eG_{\ant}}}:\overline{G_{\aff}\cap
G_{\ant}} \to \overline{e(G_{\aff}\cap
G_{\ant})}$
is a surjective morphism of algebraic monoids, it follows from
\cite[Corollary 3.9]{Re05} that
\[
E\bigl(\overline{(eG)_{\aff}\cap
(eG)_{\ant}}\bigr)= eE\bigl(\overline{G_{\aff}\cap
G_{\ant}}\bigr),
\]

Since $e\in eE\bigl(\overline{G_{\aff}\cap
G_{\ant}}\bigr)$, it follows that $eM$ is stable if and only if $ef=e$
for all $f\in E\bigl(\overline{G_{\aff}\cap
G_{\ant}}\bigr)$.

Let now $G_{\aff}\cap G_{\ant}=TU$ be a Levi decomposition. Since
$G_{\aff}\cap G_{\ant}$ is
commutative, it follows that
$E\bigl(\overline{G_{\aff}\cap
G_{\ant}}\bigr) =E\bigl(\overline{T}\bigr)$. Indeed, if follows from
\cite[Prop.~3.13]{Re05} that  $E\bigl(\overline{G_{\aff}\cap
G_{\ant}}\bigr)  =\bigcup_{g\in G_{\aff}\cap
G_{\ant}} gE(\overline{T})g^{-1}=E(\overline{T})$.
Let $f_0\in \overline{T}$ the unique
minimal idempotent; i.e.\ $f_0$ is unique idempotent in the Kernel of
the affine toric variety $\overline{T}$ --- the existence of $f_0$
follows from the fact that there exists a finite number of idempotents
elements on $\overline{T}$.   It is clear that $ef_0=e$
if and only if $ef=e$ for all $f\in  E\bigr(\overline{G_{\aff}\cap
  G_{\ant}}\bigr)$.
\end{proof}

\begin{remark}
Observe that, if we keep the notations of the proof of Proposition
\ref{prop:existseMstab}, then $f_0$ is the
unique idempotent in $E\bigl(\overline{G_{\ant}}\bigr)$ such that $f_0M$ is an
stable monoid. Moreover, $f_0$
is the maximum central idempotent for which $f_0M$ is a stable monoid.
\end{remark}

We now characterize normal, anti-affine, algebraic monoids.

\begin{theorem}
\label{thm:TheTheorem}
Let $M$ be a normal algebraic monoid. Then $M$ is anti-affine  if and only
if $eM$ is an anti-affine algebraic group, where $e$ is the unique
minimal idempotent of $\overline{G_{\ant}}$ (equivalently, $e$ is the unique
minimal idempotent of $\overline{G_{\aff}\cap G_{\ant}}$).
In particular, $e$ is the \emph{minimum} idempotent of $M$.
\end{theorem}

\begin{proof}
Firs of of, we recall that $E\bigl(\overline{G_{\aff}\cap
  G_{\ant}}\bigr)=E\bigl(\overline{G_{\ant}}\bigr)$ (see Remark
\ref{rem:idemGant}) and that  $\cO(M)=\cO(eM)$ (see Corollary
\ref{coro:thmcrucial}). Since by Proposition
\ref{prop:existseMstab} $eM$ is an stable algebraic monoid, if follows
from Corollary \ref{coro:stableantiaff} that $\cO(eM)=\Bbbk$ if and
only if $eM$ is an anti-affine algebraic group.

Finally, since $eM$ is a group, it follows that $E(eM)=e$, and hence
$ef=e$ for $f\in E(M)$.
\end{proof}

The following examples indicate why, in Theorem
\ref{thm:TheTheorem}, one must focus on the
idempotents of $E\bigl(\overline{G_{\aff}\cap
G_{\ant}}\bigr)$, rather than just any central idempotent.

\begin{example}
(1)  Let $N$ be an affine monoid  of dimension
$\dim N\geq  1$ with zero element $0_N$.  Then $0_N$ is a central idempotent of $N$.
Let $H$ be any anti-affine algebraic group. Then $M=N\times
H$ is such that $\cO(M)=\cO(N)$, while $0_NM=H$ is an
anti-affine algebraic group. Observe that $0_N$ is the minimum idempotent
of $M$.

Here, $M_{\aff}=N\times H_{\aff}$,  $G(M)_{\aff}=G(N)\times H_{\aff}$ and
$G(M)_{\ant}= \{1\}\times H_{\ant}$. Hence $G(M)_{\aff}\cap
G(M)_{\ant}=\{1\}\times
H_{\aff}$, and thus $0\notin E\bigl(\overline{G(M)_{\aff}\cap
G(M)_{\ant}}\bigr)$.

\noindent (2) Let $N$ be an irreducible affine  monoid, with
zero element $0_ N$, such that $0_N\notin
  \overline{\mathcal Z(G)}$ (take for example $N$ as in example
\ref{exam:center}). Let $M$ be  an algebraic monoid such that
$M_{\aff}=N$.
Then, $0_NM=0_NG(M)_{\ant}\cong
  G(M)_{\ant}/\bigl(G(M)_{\ant}\bigr)_{0_N}$  is an anti-affine
  algebraic group (see  for example \cite[Lemma
  1.3]{kn:brionantiaff}), whereas
  $\cO(M)=\cO(N)^{G(N)\cap G(M)_{\ant}}$ is not
  necessarily equal to the field $\Bbbk$.
\end{example}

We now show that if $M$ is anti-affine, then
$\ell_e:M\to eM$ is Serre's universal morphism from  the
pointed
variety $(M,e)$ into a commutative algebraic group.  See
\cite[Thm.~8]{Se58} and  \cite[\S
2.4]{kn:brionantiaff} for some basic properties of this morphism.

\begin{theorem}
\label{thm:genalb}
Let $M$ be a normal anti-affine algebraic monoid, and let
$e\in E(M)$ be its minimum idempotent. Then $\ell_e:M\to eM$ is
Serre's universal morphism from the pointed variety $(M,e)$ into a
commutative algebraic group. In particular, Serre's morphism fits into
the following
short exact sequence of algebraic monoids.
\begin{center}
\mbox{
\xymatrix{
1\ar@{->}[r]&M_e\ar@{->}[r]& M\ar@{->}[r]& eM\ar@{->}[r]& 1
}
}
\end{center}
\end{theorem}

\begin{proof}
Let $\sigma:M\to S$ be Serre's universal morphism from the pointed
variety $(M,e)$ into a
commutative algebraic group.
Since $M$ is anti-affine, it follows that $S$ is necessarily an
anti-affine algebraic group (see for example \cite[\S
2.4]{kn:brionantiaff}).  Moreover, it follows
from  Theorem
\ref{thm:TheTheorem} that $eM$ is an anti-affine algebraic group.
 Thus we have a commutative diagram
\begin{center}
\mbox{
\xymatrix{
M\ar@{->}[r]^-{\ell_e}\ar@{->}[d]_-{\sigma}& eM\\
S\ar@{->}[ur]_-{\varphi}&
}
}
\end{center}
\noindent Since $\varphi(0_S)=e$, it follows that $\varphi$ is a morphism of
algebraic groups (see for example
\cite[Lem.1.5]{kn:brionantiaff}). Consider the associated short exact
sequence
\begin{equation}
\label{eqn:seq}
\hfil
\xymatrix{
0\ar@{->}[r]& N\ar@{->}[r]& S\ar@{->}[r]^-{\varphi}& eM\ar@{->}[r]& 0
}
\hfill
 \end{equation}

\noindent Since $eM$ is anti-affine with
$\sigma(e)=0_S$, it follows that $\sigma|_{_{eM}}:eM\to S$ is a
morphism of algebraic groups. Moreover, we have the following
commutative diagram
\begin{center}
\mbox{
\xymatrix{
 eM\ar@{->}[r]\ar@{->}[d]_-{\alpha_{eM}}&
 M\ar@{->>}[r]^-{\sigma}\ar@{->}[d]^-{\alpha_M}&
S\ar@{->}[d]^-{\alpha_S}\\
 \mathcal A(eM)=\mathcal A(G)\ar@{->}[r]& \mathcal A(M)=\mathcal
 A(G)\ar@{->>}[r]_-{\alpha_\sigma}&
\mathcal A(S)
}
}
 \end{center}
\noindent Let $\gamma :\mathcal A(eM)\to \mathcal A(S)$ be the composition of the
horizontal arrows. Then $\gamma$ is a surjective morphism of algebraic
groups. In particular,  $\dim \mathcal A(S)\leq \dim \mathcal
 A(eM)$.

On the other hand, $\varphi\circ\sigma=\ell_e : M\to eM$,
and thus   $\sigma|_{_{eM}}$ is a splitting for the exact
sequence \eqref{eqn:seq}. It follows that $S\cong eM\times N$, and we
have the
following commutative diagram
\begin{center}
\mbox{
\xymatrix{
0\ar@{->}[r]& eM\ar@{->}[r]\ar@{->}[d]_-{\alpha_{eM}}& S=eM\times
N\ar@{->}[r]\ar@{->}[d]^-{\alpha_S}&
N\ar@{->}[r]\ar@{->}[d]^-{\alpha_N}&
0\\
& \mathcal A(eM)\ar@{->}[r]_-{\gamma}& \mathcal A(S)\ar@{->}[r]_{\beta}&
\mathcal A(N)\ar@{->}[r]& 0
}
}
 \end{center}

 It follows that $\gamma\bigl(\mathcal A(eM)\bigr)\subset
 \beta^{-1}(0)$. Hence,
 \[
\dim \mathcal A(S)\leq \dim \mathcal
 A(eM)\leq \dim \mathcal A(S)-\dim \mathcal A(N).
\]
where the second inequality follows from Chevalley's theorem on the
dimension of the
 fibers of a morphism. Thus, equality holds and $\dim \mathcal
 A(N)=0$. It follows that $N$
is an irreducible affine algebraic group. Since $\Bbbk =\cO(S)\cong
\cO(eM)\otimes \cO(N)\cong \cO(N)$, it follows that $N$ is a point, and
thus $S\cong eM$.
\end{proof}

We conclude this paper by extending the Rosenlicht decomposition of
algebraic groups (\cite[Prop.~3.1]{kn:brionantiaff}), to the setting
of algebraic monoids.

\begin{theorem}
\label{thm:rosenmon}
Let $M$ be a normal algebraic monoid, with unit group $G$, and let
$e\in E\bigl(\overline{G_{\ant}}\bigr)$ be the
minimum idempotent of $\overline{G_{\ant}}$.  Assume
that $\cO(M)$ is finitely generated. Then $M_eG_{\ant}$ is an
anti-affine algebraic monoid, and  the sequence
\begin{center}
\mbox{
\xymatrix{
1\ar@{->}[r]& M_eG_{\ant}\ar@{->}[r]& M\ar@{->}[r]_-{\varphi}&
\Spec\bigl(\cO(M)\bigr)
}
}
 \end{center}
is an exact sequence of algebraic monoids, with $\varphi$ a dominant
morphism.
\end{theorem}

\begin{proof}
 Since $\cO(M)$ is finitely
generated, it follows that $N=\Spec\bigl(\cO(M)\bigr)$ is an algebraic
monoid. Moreover, the canonical morphism  $\varphi: M\to N$ is a
morphism of algebraic monoids. If we let $S=\varphi^{-1}(1)$, then
we have the following commutative diagram. The top row is exact and the bottom row
is left exact.
\begin{center}
\mbox{
\xymatrix{
1\ar@{->}[r]& G_{\ant}\ar@{->}[r]\ar@{->}[d]&
G\ar@{->}[r]\ar@{^(->}[d]& \Spec\bigl(\cO(G)\bigr)\ar@{->}[d]\ar@{->}[r]&1\\
1\ar@{->}[r]& S\ar@{->}[r]& M
\ar@{->}[r]_-{\varphi}&N=\Spec\bigl(\cO(M)\bigr)
}
}
 \end{center}
Since $M=G_{\ant}M_{\aff}\cong G_{\ant}*_{G_{\aff}\cap
  G_{\ant}}M_{\aff}$ (Proposition \ref{prop:resultsM}) and
$\cO(M)\cong \cO(M_{\aff})^{G_{\aff}\cap
  G_{\ant}}$ (Theorem \ref{thm:crucial}),  it follows that the
restriction morphism $\varphi|_{M_{\aff}}$
is induced by the inclusion $\cO(M_{\aff})^{G_{\aff}\cap
G_{\ant}}\subset \cO(M_{\aff})$. In particular,
$\varphi|_{M_{\aff}}$ is dominant and the following diagram is
commutative
\begin{center}
\mbox{
\xymatrix{
1\ar@{->}[r]& S\ar@{->}[r]& M
\ar@{->}[r]^-{\varphi}&N=\Spec\bigl(\cO(M)\bigr)  \\
& G_{\aff}\cap G_{\ant}\ar@{^(->}[r]\ar@{^(->}[u] &
M_{\aff}\ar@{->}[r]\ar@{^(->}[u]&
N=\Spec\bigl(\cO(M_{\aff})^{G_{\aff}\cap G_{\ant}}\bigr)\ar@{=}[u]
}
}
 \end{center}
Assume that $M$ is a stable algebraic monoid. Then $G_{\aff}\cap
G_{\ant}$ is a commutative closed normal subgroup of $M_{\aff}$, and
hence it follows from Proposition \ref{prop:partialconv} that $
G_{\aff}\cap G_{\ant}$ is observable in $M_{\aff}$.  It follows from
\cite[Thm.3.18]{kn:oaoag} that the
\emph{affinized quotient} $\varphi|_{M_{\aff}}: M_{\aff}\to
\Spec\bigl(\cO(M_{\aff})^{G_{\aff}\cap
G_{\ant}} \bigr)=N$ is such that there exists an open subset $U\subset
N$ so that $(\varphi|_{M_{\aff}})^{-1}(u)$ is a closed $G_{\aff}\cap
G_{\ant}$-orbit for all
$u\in U$. We claim that this implies that
$(\varphi|_{M_{\aff}})^{-1}(1_N)=G_{\aff}\cap  G_{\ant}$. Indeed,
observe that since $G_{\aff}\cap
G_{\ant}$ is normal in $M$, it follows that  $\varphi|_{M_{\aff}}$ is
$(G\times
G)$-equivariant. It suffices now to recall that $M=G_{\ant}M_{\aff}$,
and thus
\[
\varphi^{-1}(1_N)=G_{\ant}(\varphi|_{M_{\aff}})^{-1}(1_N)=
G_{\ant}.
\]
Since $M$ is stable, $e=1$, and thus $M_eG_{\ant}=G_{\ant}$.

If $M$ is not a stable monoid, let $e\in
E\bigl(\overline{G_{\ant}}\bigr)$ be the minimum
idempotent. Then by Corollary \ref{coro:thmcrucial} it follows that
$\cO(M)=\cO(eM)$. One then concludes from (the proof of)
Corollary \ref{coro:thmcrucial}, that $\varphi|_{eM}:eM\to N\cong
\Spec\bigl(\cO(eM)\bigr)$  is the affinization morphism of
$eM$. Recalling that $G(eM)_{\ant}=eG_{\ant}$, it follows
that $eM$ is stable that we have an exact sequence
  \begin{center}
\mbox{
\xymatrix{
1\ar@{->}[r]& eG_{\ant}\ar@{->}[r]& eM
\ar@{->}[r]_-{\varphi}&\Spec\bigl(\cO(eM)\bigr)
}
}
 \end{center}
that fits into the following commutative diagram of algebraic monoids, where
the vertical sequence in the center is exact.
  \begin{center}
\mbox{
\xymatrix{
&                                    & 1\ar@{->}[d]             &   \\
&                                    & M_e  \ar@{->}[d]\ar@{^(->}[dl]   & \\
1\ar@{->}[r]& S\ar@{->}[r]\ar@{->}[d]^-{\ell_e}&
M\ar@{->}[r]^-{\varphi}\ar@{->}[d]^-{\ell_e}&
\Spec\bigl(\cO(M)\bigr)\ar@{=}[d]\\
1\ar@{->}[r]& eG_{\ant}\ar@{->}[r]& eM
\ar@{->}[r]_-{\varphi}\ar@{->}[d]&\Spec\bigl(\cO(eM)\bigr)\\
& & 0&
}
}
 \end{center}
 It follows that
$S=\ell_e^{-1}(eG_{\ant})=M_eG_{\ant}$.

To complete the proof we observe that $G(S)=G_eG_{\ant}$, with
$G_e\subset G_{\aff}$. Hence
$ G(S)_{\ant}=G_{\ant}$. Since $e$ is the minimum idempotent of the
closure $\overline{G_{\ant}}\subset M$, it follows that $e$ is also
the minimum idempotent of the closure $\overline{G_{\ant}}\subset
M_eG_{\ant}$. It suffices now to observe that
$e(M_eG_{\ant})=eG_{\ant}$, and then apply Theorem \ref{thm:TheTheorem}.
\end{proof}

\bigskip

\bigskip

\begin{scriptsize}
\noindent \begin{tabular}{ll}
{\sc Lex Renner}   \vspace*{2pt} &          {\sc Alvaro Rittatore} \\
University of Western Ontario \hspace*{3.5cm} &    Facultad de Ciencias\\
London, N6A 5B7,  Canada        &  Universidad de la Rep\'ublica\\
{\tt lex@uwo.ca}               &  Igu\'a 4225\\
                &      11400 Montevideo, Uruguay\\
 & {\tt alvaro@cmat.edu.uy}
\end{tabular}
\end{scriptsize}


\begin{thebibliography}{100}

\bibitem{kn:bb-induced} A.~Bialynicki-Birula, \emph{On induced actions
  of algebraic groups}, Ann. de l'Inst. Fourier, 43 no. 2 (1993),
   365-368.

\bibitem{kn:obsdef}
A.~Bialynicki-Birula, G.~Hochschild, G.D.~Mostow, \emph{Extensions of
  representations of algebraic linear groups}, Amer. J. Math. 85
(1963), 131-–144.

\bibitem{Br06} M.~Brion,
\emph{Log homogeneous varieties},
Actas del XVI Coloquio Latinoamericano de \'Algebra, 1--39,
Revista Matem\'atica Iberoamericana, Madrid, 2007;
arXiv: math/0609669.


\bibitem{kn:brionlocal} M.~Brion, \emph{The local structure of algebraic
    monoids}, preprint,  arXiv:0709.1255 [math.AG].

\bibitem{kn:brionantiaff} M.~Brion, \emph{Anti-affine algebraic
      groups}, preprint, arXiv:0710.5211 [math.AG].


\bibitem{ritbr} M.~Brion, A.~Rittatore
\emph{The structure of normal algebraic monoids},  Semigroup Forum  74
(2007),  no. 3, 410--422; arXiv: math.AG/0610351.

\bibitem{Co02} B.~Conrad,
\emph{A modern proof of Chevalley's theorem on algebraic groups},
J. Ramanujan Math. Soc. {\bf 17} (2002), 1--18.

\bibitem{DG70} M.~Demazure and P.~Gabriel,
{\it Groupes alg\'ebriques}, North Holland, Amsterdam, 1970.

\bibitem{fer-ritt} W.~Ferrer-Santos, A.~Rittatore, \emph{Actions and
Invariants of Algebraic Groups.}
Series: Pure and Applied Math., 268, Dekker-CRC Press,
Florida, (2005).

\bibitem{kn:Gross-14}
F.D.~Grosshans, \emph{Observable groups and {H}ilbert's fourteenth problem},
  Amer. J. Math. \textbf{95} (1973), 229--253.






\bibitem{Pu88} M.~S.~Putcha,
\emph{Linear Algebraic Monoids},
London Math. Soc. Lecture Notes Series {\bf 133},
Cambridge University Press, Cambridge, 1988.

\bibitem{Re05} L.~E.~Renner,
\emph{Linear Algebraic Monoids},
Encyclop\ae dia of Mathematical Sciences {\bf 134},
Invariant Theory and Algebraic Transformation Groups, V,
Springer-Verlag, Berlin, 2005.

\bibitem{kn:oaoag} L.~E.~Renner, A.~Rittatore, \emph{Observable actions
    of algebraic groups}, preprint, arXiv:0902.0137v2 [math.AG].


\bibitem{Ri98} A.~Rittatore,
\emph{Algebraic monoids and group embeddings},
Transformation Groups {\bf 3}, No.~4 (1998), 375--396;
arXiv: math.AG/9802073.

\bibitem{Ri06} A.~Rittatore,
\emph{Algebraic monoids with affine unit group are affine},
Transform. Groups {\bf 12} (2007), 601--605; arXiv: math.AG/0602221.

\bibitem{Ro56} M.~Rosenlicht, \emph{Some basic theorems on algebraic
groups}, Amer.~J.~Math. {\bf 78} (1956), 401--443.

\bibitem{Se58} J.-P.~Serre,
{\it Morphismes universels et vari\'et\'e d'Albanese},
S\'eminaire Chevalley (1958--1959), Expos\'e No. 10,
Documents Math\'ematiques {\bf 1}, Soc. Math. France, Paris, 2001.

\bibitem{So93} L.~Solomon, \emph{An introduction to reductive monoids},
in Semigroups, formal languages and groups, York,
1993, NATO Adv. Sci. Inst. Ser. C Math. Phys. Sci., 466, 1995, 295--352.

\bibitem{Ti06} D.~A.~Timashev,
{\it Homogeneous spaces and equivariant embeddings}, to appear in the
Encyclop\ae dia of Mathematical Sciences, subseries Invariant Theory
and Algebraic Transformation Groups;  arXiv: math.AG/0602228.
\end{thebibliography}
\end{document}